\documentclass[12pt]{elsarticle}
\usepackage{amssymb}
\usepackage{amsthm}
\usepackage{enumerate}
\usepackage{amsmath,amssymb,amsfonts,amsthm,fancyhdr}
\usepackage{mathrsfs}
\usepackage[colorlinks,linkcolor=black,anchorcolor=black,citecolor=black,CJKbookmarks=True]{hyperref}
\usepackage{tikz}
\usepackage{mathabx}

\newdefinition{rem}{Remark}[section]
\newdefinition{theorem}{Theorem}[section]
\newdefinition{corollary}{Corollary}[section]
\newdefinition{definition}{Definition}[section]
\newdefinition{lemma}{Lemma}[section]
\newdefinition{prop}{Proposition}[section]
\numberwithin{equation}{section}

\begin{document}
\begin{frontmatter}
\title{Fourier decay rate of coin-tossing type measures}
\author[a]{Xiang Gao}\ead{gaojiaou@gmail.com}
\author[b]{Jihua Ma}\ead {jhma@whu.edu.cn}
\author[b]{Kunkun Song*}\ead{songkunkun@whu.edu.cn}
\author[b]{Yanfang Zhang}\ead{yanfangzhang007@whu.edu.cn}
\address[a]{Department of Mathematics, Hubei key laboratory of Applied Mathematics, Hubei University, Wuhan, 430062, P. R. China}
 \address[b]{School of Mathematics and Statistics, Wuhan University, Wuhan, 430072, P. R. China}
\cortext[cor1]{Corresponding author.}

\begin{abstract}
\par Motivated by a classical result of Hartman and Kershner, we give explicit estimates of the Fourier decay rate of some coin-tossing type measures. Using an elementary method originated from Cassels and Schmidt, we also prove that almost every point is absolutely normal with respect to such measures. As an application, we present new examples of measures whose Fourier decay rate could be as slow as possible, yet almost all points are absolutely normal, which complements a result of R. Lyons on the set of non-normal numbers.
\end{abstract}
\end{frontmatter}
\section{Introduction}
\indent Given a Borel probability measure $\mu$ on $\mathbb{R}$, its Fourier-Stieltjes transform is defined by
\begin{equation}\label{1.1}
{\widehat\mu}(t)=\int_{\mathbb{R}}e^{2\pi itx}d\mu(x),\ \ \ \ t\in\mathbb{R}.
\end{equation}
The decay rate of ${\widehat\mu}(t)$ as $|t|\rightarrow\infty$ often provides lots of information about $\mu$. For instance, the absolute continuity and singularity of the measure $\mu$, the equidistribution phenomenon \cite{JS16} and Roth-type theorem \cite{LP09} on the support of the measure $\mu$. It is well known that if the Fourier decay rate of a measure is polynomial, that is $|{\widehat\mu}(t)|= O(|t|^{-\epsilon})$ for some positive constant $\epsilon$, then the Hausdorff dimension of $\mu$ is no less than $2\epsilon$. Moreover, $\mu$ almost every point is absolutely normal (normal to any integer base). In fact, to ensure the normality, the condition of polynomial decay can be weakened to logarithmic decay, see \cite{MS17}. So one would ask whether there exist some measures which have slower Fourier decay rate, yet to guarantee that $\mu$ almost every point is absolutely normal. We shall show that the answer is positive by providing a class of explicit examples.\\
\indent  Estimates for Fourier decay rate of measures are very useful tools in Fourier analysis, dynamical systems and fractal geometry. However, it is often not easy to obtain such estimate. Sometimes it may also happen that ${\widehat\mu}(t)$ does not vanish at infinity, while its asymptotic average of its Fourier transform tends to $0$. This has been proved in $1939$ by Wiener and Wintner for the standard middle-third Cantor measure, and generalized by Strichartz to more general self-similar measures, see \cite{Bis04,Stri}.\\
\indent One of the important examples which can be explicitly estimated are Bernoulli convolutions $\mu_\lambda$, i.e., the distribution of the random sum $\sum \pm \lambda^n$, where the signs are chosen independently with probability $\frac{1}{2}$. It is Kershner and Bary (independently) who first showed that the Fourier decay rate has close relation with the arithmetic property of $\lambda$ \cite{Ker36}. More precisely, the Fourier decay rate of $\mu_\lambda$ is logarithmic when $\lambda$ is a rational number. For similar results about some classes of algebraic numbers, see \cite{BS14,DFW07,GM17}. Erd\H{o}s \cite{Er39,Er40} proved that the decay rate of ${\widehat\mu_\lambda}(t)$ is polynomial for a.e. $\lambda$ sufficiently close to 1. For more information in this direction, see the survey \cite{PSS00}.\\
\indent In recent years, much attention has been paid to the Fourier decay of various measures. For instance, the large deviation result on the Fourier transform of self-similar measures has been obtained in \cite{Tsu}. The behaviour of Fourier decay rate of measures which are images of smooth perturbation of self-similar measures has been studied in \cite{CG17,MS17}. See also the original work of Kaufman \cite{Kau84}.
Very recently, by using the discretized sum-product theorem on $\mathbb{R}$, Bourgain and Dyatlov \cite{BD17} proved that the Patterson-Sullivan measure on the limit set of a co-compact subgroup of $SL(2,\mathbb{Z})$ has polynomial decay rate. For similar results about Furstenberg measures on $SL(2,\mathbb{R})$, see \cite{Li17}. For Gibbs measures of the Gauss map, see \cite{JS16,Kau81}.\\
 \indent In this paper, we are concerned with the coin-tossing type measure which is an infinite convolution on $(0,1)$:
\begin{equation}\label{1.2}
 \mu=\Asterisk_{n=1}^{\infty}\left(\frac{1}{2}(1+\phi(n))\delta_0+\frac{1}{2}(1-\phi(n))\delta_{2^{-n}}\right),
 \end{equation}
 where $\phi:\mathbb{N}\rightarrow(0,1)$ is a weight function, and $\delta_x$ is the Dirac measure at $x$. There is an alternative way to interpret the measure $\mu$ from the probabilistic viewpoint. Consider the random sum
$$S=\sum_{n=1}^{\infty}X_n 2^{-n},$$
where $\left\{X_n\right\}_{n\geq1} $ are the independent random variables satisfying
$$P(X_n=0)=\frac{1}{2}(1+\phi(n)),\quad P(X_n=1)=\frac{1}{2}(1-\phi(n)).$$
In fact, $X_n$ can be interpreted as the random variable, which equals one if the outcome of the $n$-th coin toss is head (with probability $\frac{1}{2}(1-\phi(n))$), and zero otherwise in an infinite coin-tossing experiment.
Then $\mu$ is the distribution of $S$, i.e., $\mu(A)=P(\omega: S(\omega)\in A)$.\\
\indent The classical Jessen-Wintner law of pure types \cite{PSS00} implies that $\mu$ is either absolutely continuous or singular. Salem \cite{Sal43} proved that the above coin-tossing type measure is singular if and only if the series $\sum \phi^2(n)$ is divergent. Moreover, he showed that $\phi(n)\rightarrow0$ is a necessary condition for the Fourier transform of $\mu$ vanish at infinity (in this case, $\mu$ is called a Rajchman measure). In 1971, Blum and Epstein \cite{BE71} proved that $\phi(n)\rightarrow0$ is also a sufficient condition of $\mu$ being a Rajchman measure, for details, the reader could see \cite{M71}.\\
\indent So there is a natural question: what is the  explicit Fourier decay rate of the coin-tossing type measure in \eqref{1.2} for general weight function $\phi$? Observe that when $\phi(n)\equiv 0$, $\mu$ degenerates into the Lebesgue measure. Heuristically speaking, the more rapidly  $\phi$ decreases to $0$, the more rapidly the Fourier transform ${\widehat\mu}(t)$ tends to $0$. We shall show that this is indeed the case. As far as we know, there is no other results having been known till now except the one  due to Hartman and Kershner \cite{HarKer}, who proved that $|{\widehat\mu}(t)|=O(\log^{-\frac{1}{2}}|t|)$ when $\phi(n)=n^{-\frac{1}{2}}$.
Using Hartman-Kershner method, we shall present explicit decay rates for general weight functions. Our main result is the following.
\begin{theorem}\label{thm1.1}
Suppose $\phi:\mathbb{N}\rightarrow(0,1)$ is a monotone function decreasing to zero. Let $\mu$ be the corresponding measure defined by \eqref{1.2}.
 \begin{enumerate}[(1)]
   \item If $\dfrac{\phi(n)}{\phi(n+1)}<2$ for any $n$, then $|\widehat{\mu}(t)|= O(\phi(\lceil\gamma_1\log_2|t|\rceil))$, where $\gamma_1\in(0,1)$ is a constant depending only on $\phi$;
   \item If $\dfrac{\phi(n)}{\phi(n+1)}\geq2$ for any $n$, then $|\widehat{\mu}(t)|=O({|t|^{-\gamma_2}})$, where $\gamma_2\in(0,1)$ is a constant depending only on $\phi$.
 \end{enumerate}
\end{theorem}
\begin{rem}
Before proceeding on, we give some remarks.
\begin{enumerate}[$\bullet$]
\item Typical examples of weight functions $\phi$ in case $(1)$ are $(\log n)^{-\tau_1}, n^{-\tau_2}$, where $\tau_1,\tau_2 $ are positive constants, while for case $(2)$ are $\kappa^{-n}$, with $\kappa$ being constant no smaller than $2$.
\item One may wonder if the sequence $\{2^{-n}\}_{n\geq1}$ could be replaced by $\{a^{-n}\}_{n\geq1}$ for some integer $a\geq3$.
We shall see that the corresponding measure is not a Rajchman measure.
\item Our proof of Theorem \ref{1.1} is based on the behaviour of the fractional part of $\{2^{-n}t\}_{n\geq1}$.
\end{enumerate}
\end{rem}
\par Following the method of Cassels \cite{Cass} and Schmidt \cite{Sch60}, we are able to prove the following theorem.
\begin{theorem}\label{thm1.2}
Let $\mu$ be the measure defined by \eqref{1.2}. Then $\mu$ almost every point is absolutely normal.
\end{theorem}
\par Recall that a real number $x\in [0,1)$ is said to be normal to an integer base $b\geq 2$, if $x=0.x_1x_2\cdots$ is expanded in base $b$ with that every combination of digits occurs with the proper frequency. Equivalently, $\{b^n x\}_{n\geq1}$ is uniformly distributed modulo one. Further, if $x$ is normal in any base $b$, then $x$ is called absolutely normal. As a fundamental tool, the famous Weyl's criterion allows  equidistribution questions to be reduced to estimate on exponential sums. Davenport, Erd\H{o}s and Leveque \cite{DEL63} proved the following criterion:
 \begin{theorem}[Davenport-Erd\H{o}s-Leveque]\label{del}
  Let $\mu$ be a probability measure on $[0,1]$ and $\{s_k\}_{k\geq1}$ a sequence of positive integers with strictly increasing. If for any $h\in\mathbb{Z} \setminus\{0\}$,
$$\sum\limits^{\infty}_{N=1}\frac{1}{N^3}\sum\limits^{N}_{n,m=1}\widehat{\mu}(h(s_n-s_m))<\infty,$$
then for $\mu$ almost every $x$, the sequence $\{s_k x\}_{k\geq1}$ is uniformly distributed modulo one.
\end{theorem}
\par As a consequence, if
\[\sum_{n=2}^{\infty}\frac{\hat{\mu}(n)}{n\log n}<\infty,\]
then for $\mu$\ almost every $x$, $\{n_kx\}_{k\geq1}$ is uniformly distributed modulo one, whenever $\{n_k\}_{k\geq1}$ is a lacunary integer sequence. On the other hand, Lyons constructed a measure $\nu$ for which the above series is divergent (under some mild additional condition), and $\nu$ almost every point is not normal in base $2$, for the details see \cite{Ly}.
Motivated by the classical Borel's normal numbers theorem, i.e., almost all number in $[0,1)$ are absolutely normal with respect to Lebesgue measure, in $1964$ Kahane and Salem asked whether the same is true with respect to any Rajichman measure. As was shown in the Lyons' counterexample, the answer of the Kahane-Salem problem is negative for general measures. However, for the specific coin-tossing measures, we shall show that the answer for the  Kahane-Salem problem is positive.
\par Combining Theorem \ref{thm1.1} with Theorem \ref{thm1.2}, we immediately obtain the following corollary.
\begin{corollary}\label{1.1}
There exists a Rajchman measure whose Fourier decay tends to zero as slowly as possible, and for which almost every point is absolutely normal.
\end{corollary}
\begin{rem}By constructing a variant of coin-tossing type measure, Lyons stated that there exists a Rajchman measure which support on the set of non-normal numbers in base $2$. While we could find a Rajchman measure with explicit Fourier decay which support on the set of the normal numbers in all bases. In some sense our result could be seen as a complementary result of Lyons.
\end{rem}
\section{Preliminaries}
\indent Throughout this text, we use $\|\cdot\|$ to denote the distance to the nearest integer, $\lceil x\rceil$ the smallest integer larger than $x$, $[1,N]$ the set $\{1,2,\ldots,N\}$, $\sharp$ the cardinality of a set and $e(x)=e^{2\pi ix}$ respectively.\\
\indent One may wonder why we choose the sequence $\{2^{-n}\}_{n\geq1}$ in the construction of the measure $\mu$, see \eqref{1.2}. The following proposition partially explains the reason.
\begin{prop}\label{prop2.2}
Let $a\geq3$ be an integer and
\[\mu_a=\Asterisk_{n=1}^{\infty}\left(\frac{1}{2}(1+\phi(n))\delta_0+\frac{1}{2}(1-\phi(n))\delta_{a^{-n}}\right).\]
Then $\mu_a$ is not a Rajchman measure.
\end{prop}
\begin{proof}
By the multiplication rule of Fourier transform of measures, it is easy to see that
\[\widehat{\mu_a}(t)=\prod\limits_{n=1}^{\infty}\left(\frac{1}{2}(1+\phi(n))+\frac{1}{2}(1-\phi(n))e(a^{-n}t)\right).\]
Hence \begin{eqnarray*}
    |\widehat{\mu_a}(t)|^2&=&\prod\limits_{n=1}^{\infty}\left(\frac{1}{2}(1+\phi^2(n))+\frac{1}{2}(1-\phi^2(n))\cos{2\pi a^{-n}t}\right)\\
     &=&\prod\limits_{n=1}^{\infty}\left(\phi^2(n)\sin^2(\pi a^{-n}t)+\cos^2(\pi a^{-n}t)\right)\\
     &\geq&\prod\limits_{n=1}^{\infty}\cos^2(\pi a^{-n}t).
   \end{eqnarray*}
  Choosing $t$ along the sequence $(a^{k})_{k\geq 1}$,
   \begin{eqnarray*}
     |\widehat{\mu_a}(a^{k})|^2 &\geq&\prod\limits_{n=1}^{\infty}\cos^2{\pi a^{-n+k}}=\prod\limits_{n=1}^{\infty}\cos^2{a^{-n}\pi}\geq \cos^2{\frac{\pi}{a}}\prod\limits_{n=2}^{\infty}\cos^2{\frac{\pi}{2^n}}=\frac{4}{\pi^2}\cos^2{\frac{\pi}{a}}>0.\\
\end{eqnarray*}
   \end{proof}
 \par By virtue of the above proposition, we only consider the case $a=2$ in what follows. The following equation will be repeatedly used
\begin{equation}\label{dk}
\widehat{\mu}(t)=\prod\limits_{n=1}^{\infty}\left(\frac{1}{2}(1+\phi(n))+\frac{1}{2}(1-\phi(n))e(2^{-n}t)\right).
\end{equation}
So
\begin{equation}\label{dks}
|\widehat{\mu}(t)|^2 =\prod\limits_{n=1}^{\infty}(\phi^2(n)\sin^2{\pi2^{-n}t}+\cos^2{\pi2^{-n}t}).
\end{equation}
In order to estimate the Fourier decay rate, we define
 \[A_{\phi}:=\left\{t:\ \parallel t\parallel<2^{-K_{\phi}}\right\},\] where $K_{\phi}$ is a positive number which depends on $\phi$.
\par Let \[R_n=\left\{t\in(0,\infty):\ |t-\frac{k}{2}|<2^{-K_{\phi}-n},\text{ for some odd $k$ }\right\}\]
be the set of all points which are at a distance at least $\frac{1}{2}-2^{-K_{\phi}-n}$ from the nearest integer. Clearly if the point $t$ falls into the set $R_n$, the value of $\cos^2{\pi t}$ is small. We now have the following easy lemma.
 \begin{lemma}\label{gx}
\begin{enumerate}[1)]
\item If $t\notin A_{\phi}$, then there exists a positive $\delta=\delta_{\phi}<1$ such that
 \[|\phi^2(n)\sin^2 \pi t+\cos^2 \pi t|<\delta<1, \text{ for any}\ n\geq 1.\]
\item If $t\in R_n$, then $$|\cos \pi t|\leq \pi2^{-n-K_{\phi}}.$$
\end{enumerate}
\end{lemma}
 \begin{proof}
 \begin{enumerate}[1)]
 \item It is obvious that
 \begin{eqnarray*}
|\phi^2(n)\sin^2 \pi t+\cos^2 \pi t|&=&|1-(1-\phi^2(n))\sin^2 \pi t|\\
  &\leq&1-(1-\phi^2(1))\sin^2\pi 2^{-K_{\phi}}\\
  &=&\delta<1.
 \end{eqnarray*}

 \item by the definition of $R_n$,

$$|\cos \pi t| \leq |\sin(\pi2^{-K_{\phi}-n})|\leq \pi2^{-K_{\phi}-n}.$$
 \end{enumerate}
 \end{proof}

\par Now let $t\geq2$ be fixed, suppose that $m\geq1$ is the unique integer such that
\[2^{m}\leq t<2^{m+1}.\]
Consider the set
$$\mathscr{A}:=\{1\leq n\leq m: 2^{-n}t\in A_{\phi}\},\quad k:=\sharp(\mathscr{A}).$$
That is, there exist exactly $k\ (k\leq m)$ indices in $\mathscr{A}$ for which $2^{-n}t\in A_{\phi}$ among the indices of $1,2,\cdots,m$.
 We split $\mathscr{A}$  into a disjoint union  $\bigcup\limits_{i=1}^{j}B_i$, where
 $B_i:=\{n_i-l_i, n_i-l_i-1,\cdots, n_i-1\}$,
being a block of $l_i$ consecutive integers, and
    \begin{equation}\label{2.1}
\sum\limits_{i=1}^{j}l_i=k.
\end{equation}
Note that $n_i$ (the $i$-th good index) is the first index which does not fall into $A_{\phi}$ among the integers between $B_i$ and $B_{i+1}$.
Obviously,
\begin{equation}\label{2.2}
n_i\geq\sum\limits_{s=1}^{i}l_s+1.
\end{equation}
 The following combinatorial lemma is key in the proof of Theorem \ref{thm1.1}.
\begin{lemma}\label{lem2.2}
Let $n_i, l_i$ be as above. Then we have $2^{-n_i}t\in R_{l_i}$.
\end{lemma}
\begin{proof}
 By the definition of  $B_i$ and $n_i$,
\[2^{-n_i}t\notin A_{\phi}, 2^{-(n_i-1)}t\in A_{\phi}, 2^{-(n_i-2)}t\in A_{\phi},\cdots, 2^{-(n_i-l_i)}t\in A_{\phi}.\]
Note that
\begin{equation}\label{2.3}
  \left\{
    \begin{aligned}
     &|2^{-n_i}t-k_0|\geq2^{-K_{\phi}}\\
   &|2^{-(n_i-1)}t-k_1|<2^{-K_{\phi}}\\
   &|2^{-(n_i-2)}t-k_2|<2^{-K_{\phi}}\\
   &\cdots\cdots\cdots\cdots\cdots\cdots\cdots\\
 &|2^{-(n_i-l_i)}t-k_{l_i}|<2^{-K_{\phi}}\\
    \end{aligned}
   \right.
 \end{equation}
where $k_j (0\leq j\leq l_i)$ is the integral part of $2^{-(n_i-j)}t (0\leq j\leq l_i)$ respectively,
we claim that $k_{l_i}=2^{l_i-1}k_1$.
Indeed, consider the two equations
\begin{align*}
    2^{-(n_i-1)}t=k_{1}+\varepsilon_{1} \ \ \text{and}\ \  2^{-(n_i-2)}t=k_{2}+\varepsilon_{2},
\end{align*}
where $\mid\varepsilon_{1}\mid <2^{-K_{\phi}},\mid\varepsilon_{2}\mid<2^{-K_{\phi}}$, we have
\begin{equation}\label{x}
|k_{1}-2k_{2}|=|\varepsilon_{1}-2 \varepsilon_{2}|< 3\cdot2^{-K_{\phi}}\leq 1.
\end{equation}
Since the left-hand side of \eqref{x} is an integer, it follows that
$k_{1}=2k_{2}$.
Similarly, if there are $l_i$ consecutive integers which fall into $A_{\phi}$, i.e.,
$\mid\varepsilon_{j}\mid <2^{-K_{\phi}} \text{ for all } 1\leq j\leq l_i$, then
\begin{equation}\label{2.5}
k_{l_i}=2^{l_i-1}k_1.
\end{equation}
Substituting \eqref{2.5} into the last inequality of \eqref{2.3},
  \begin{equation}\label{2.6}
   |2^{-n_i}t-k_1/2|<2^{-K_{\phi}-l_i}.
   \end{equation}
 On the other hand,
   \begin{equation}\label{2.7}
   |2^{-n_i}t-k_0|\geq2^{-K_{\phi}}.
   \end{equation}
    Combining \eqref{2.6} with \eqref{2.7}, it is easy to see that $k_1$ is odd, thus $ 2^{-n_i}t\in R_{l_i}$.
\end{proof}

\par Armed with the combinatorial Lemma \ref{lem2.2}, we have the following lemma which is useful in the estimate of the Fourier decay.
\begin{lemma}\label{lem3.1}
\[|{\widehat\mu}(t)|^2\leq\delta^{m-k-j}\prod_{i=1}^{j}(\phi^2(n_i)+\pi^22^{-2(l_i+K_{\phi})}).\]
\end{lemma}
\begin{proof}
Recall the equation \eqref{dks},
\begin{eqnarray*}
    |{\widehat\mu}(t)|^2&=&\prod\limits_{n=1}^{\infty}(\phi^2(n)\sin^2{\pi2^{-n}t}+\cos^2{\pi2^{-n}t})\\
     &\leq&\delta^{m-k-j}\prod_{i=1}^{j}(\phi^2(n_i)+\pi^22^{-2(l_i+K_{\phi})}).
\end{eqnarray*}
The reason of the last inequality is as follows:
Recall that there exist exactly $m-k$ values of $n\ (n=1,2,\cdots,m)$ such that $2^{-n}t$ is not in $A_{\phi}$. Among these indices, there are
$j$ good indices $n_k,1\leq k\leq j$. For theses good indices, by Lemma \ref{gx} and Lemma \ref{lem2.2},
\[\phi^2(n_i)\sin^2{\pi2^{-n_i}t}+\cos^2{\pi2^{-n_i}t}\leq\phi^2(n_i)+\pi^22^{-2(l_i+K_{\phi})}.\]
 The remaining $m-k-j\ (\geq0)$ factors are less than $\delta$ from Lemma \ref{gx}.
Finally, all other factors are replaced by 1.
\end{proof}
\begin{rem}
We can obtain the lower bound of the Fourier decay along some subsequence.
Let $t=2^m$, we claim that
$|{\widehat\mu}(2^m)|^2\geq\frac{4}{\pi^2}\phi^2(m+1)$.
Indeed,
\begin{eqnarray*}
  |{\widehat\mu}(2^m)|^2&=&\prod\limits_{n=1}^{\infty}(\phi^2(n)\sin^2{2^{-n+m}\pi}+\cos^2{2^{-n+m}\pi})\\
  &=&\phi^2(m+1)\prod_{n=m+2}^{\infty}(\phi^2(n)\sin^2{2^{-n+m}\pi}+\cos^2{2^{-n+m}\pi})\\
  &\geq&\phi^2(m+1)\prod_{n=2}^{\infty}\cos^2{2^{-n}\pi}=\frac{4}{\pi^2}\phi^2(m+1).
\end{eqnarray*}
\end{rem}

\par Next, in order to prove Theorem \ref{thm1.2}, we need the following lemma.
\begin{lemma}\label{lem2.1}
Let $a,b$ be two positive real numbers. Then
 \[|a+be(t)|\leq a+b-4\min\{a,b\}\|t \|^2.\]
 \end{lemma}
\begin{proof}
Without loss of generality, we may assume that $a\geq b$, otherwise, replace $t$ by $-t$.
Thus it suffices to prove $a+b-|a+be(t)|\geq 4b \|t\|^2$.
Using the basic inequality $$\sin^2 t\geq \frac{4}{\pi^2}t^2,\ \text{for all } -\frac{\pi}{2}\leq t\leq \frac{\pi}{2}.$$
Since $0\leq\| t\|\leq \frac{1}{2}$, substituting $t$ for $\pi\| t\|$, we obtain
\begin{equation}\label{bds}
\sin^2 \pi t=\sin^2 \pi \| t\|\geq 4 \| t\|^2.
\end{equation}
It is clear that
\begin{equation}\label{bds1}
(a+b+| a+be(t)|)(a+b-| a+be(t)|)=4ab\sin^2 \pi t.
\end{equation}
Since $a,b$ are positive numbers,
\begin{equation}\label{bds2}
a+b+|a+be(t)|\leq4a.
\end{equation}
By \eqref{bds} and \eqref{bds1}, we get
\begin{equation}\label{bds3}
(a+b+| a+be(t)|)(a+b-|a+be(t)|)\geq 16ab\| t\|^2\geq 4a4b\| t\|^2.
\end{equation}
Combing \eqref{bds2} with \eqref{bds3},
we complete the proof.
\end{proof}
\par Using the monotonic property of $\phi$, we obtain an immediate application.
\begin{corollary}\label{tl1}
\[|\frac{1}{2}(1+\phi(n))+\frac{1}{2}(1-\phi(n))e(2^{-n}t)|\leq 1-2(1-\phi(1))\|2^{-n}t\|^2.\]
\end{corollary}

\section{Proofs of Theorem \ref{thm1.1} and Theorem \ref{thm1.2}}

In this section, we shall prove Theorem \ref{thm1.1} and Theorem \ref{thm1.2}.\\
\textbf{Proof of Theorem \ref{thm1.1}:}
\begin{enumerate}[(1)]
\item By Lemma \ref{lem3.1},
   \begin{equation}\label{g}
 |{\widehat\mu}(t)|^2\leq\delta^{m-k-j}\prod_{i=1}^{j}(\phi^2(n_i)+\pi^22^{-2(l_i+K_{\phi})}).
  \end{equation}
  By the assumption of $\dfrac{\phi(n)}{\phi(n+1)}<2$, it follows that
    \begin{equation}\label{3.1}
  \phi(1+l_1)<2^{l_2}\phi(1+l_1+l_2).
  \end{equation}
    Since $\phi(n)$ is non-increasing, along with \eqref{2.2},
\begin{equation}\label{3.2}
  \phi(n_i)\leq\phi(1+l_1+l_2+\cdots +l_i).
  \end{equation}
 Now we estimate the right-hand side of \eqref{g} by induction on $j$. When $j=2$, we have

 \[\prod_{i=1}^{2}(\phi^2(n_i)+\pi^22^{-2(l_i+K_{\phi})})=(\phi^2(n_1)+\pi^22^{-2(l_i+K_{\phi})})(\phi^2(n_2)+\pi^22^{-2(l_i+K_{\phi})}).\]
 Let $K_{\phi}$ be large enough so that
$$\phi^2(2)+\pi^24^{-1-K_{\phi}}+\pi^24^{-K_{\phi}}\leq 1\ \text{and}\quad\pi2^{-K_{\phi}}\leq1.$$
Here, we may choose
 $$K_{\phi}=K_{\phi,1}:=\frac{1}{2}\log_{2}\frac{5\pi^2}{4(1-\phi^2(2))}.$$
In view of \eqref{3.1} and \eqref{3.2}, and using the monotonic property of $\phi$, it follows that
   \begin{eqnarray*}
  \prod_{i=1}^{2}(\phi^2(n_i)+\pi^22^{-2(l_i+K_{\phi})})&=&(\phi^2(n_1)+\pi^22^{-2(l_1+K_{\phi})})(\phi^2(n_2)+\pi^22^{-2(l_2+K_{\phi})}) \\
  &\leq&(\phi^2(2)+\pi^24^{-1-K_{\phi}}+\pi^24^{-K_{\phi}})\phi^2(1+l_1+l_2)\\
  &+&\pi^42^{-2(l_1+l_2+2K_{\phi})}\\
  &\leq&\phi^2(1+l_1+l_2)+\pi^22^{-2(l_1+l_2+K_{\phi})}.
  \end{eqnarray*}

 By induction, we  have
 \begin{eqnarray*}
  \prod_{i=1}^{j}(\phi^2(n_i)+\pi^22^{-2(l_i+K_{\phi})})&\leq&\phi^{2}(1+l_1+l_2+\cdots+l_j)+\pi^22^{-2(l_1+l_2+\cdots+l_j+K_\phi)}  \\
  &\leq&\phi^{2}(1+k)+\pi^22^{-2(k+K_\phi)}\\
  &\leq&(1+\frac{\pi^2}{4^{K_\phi}\phi^{2}(1)})\phi^{2}(k).
  \end{eqnarray*}
Hence,
\begin{equation}\label{3.3}
 |{\widehat\mu}(t)|^2\leq \delta^{m-k-j}\prod_{i=1}^{j}(\phi^2(n_i)+\pi^22^{-2(l_i+K_{\phi})})\leq C\delta^{m-k-j}\phi^{2}(k),
  \end{equation}

where $C=1+\pi^24^{-K_\phi}\phi^{-2}(1)$ is a constant depending on $\phi$.
\par In the remainder we estimate the right-hand of inequality \eqref{3.3} under the condition that
\begin{equation}\label{fg}
m-j-k\geq0\ \text{and}\ j\leq k\leq m.
 \end{equation}
Let $0<\gamma_1:=\frac{-\log_2\delta}{2}<1$, we distinguish two cases,
\begin{enumerate}[(a)]
  \item if $k\geq\gamma_1 m$, in view of \eqref{fg} and the fact that $\phi$ is non-increasing,
it is clear that
\[\delta^{m-k-j}\phi^{2}(k)\leq\delta^{0}\phi^{2}(\lceil\gamma_1 m\rceil)=\phi^{2}(\lceil\gamma_1 m\rceil).\]
  \item if $k<\gamma_1 m$, then
  \begin{equation*}
   \delta^{m-k-j}\phi^{2}(k)\leq\delta^{m-2k}\phi^{2}(k)\leq\delta^{m-k/{\gamma_1}}\phi^{2}(k)
  =\phi^{2}(\lceil\gamma_1 m\rceil)\frac{\phi^{2}(k)\delta^{m}}{\phi^{2}(\lceil\gamma_1 m\rceil)\delta^{\frac{k}{\gamma_1}}}
\end{equation*}
\end{enumerate}
On the assumption of $0<\phi(n)<2\phi(n+1)$, and the definition of $\gamma_1$,
we obtain
\begin{eqnarray*}
\frac{\phi^{2}(k)\delta^{m}}{\phi^{2}(\lceil\gamma_1 m\rceil)\delta^{\frac{k}{\gamma_1}}}&\leq&\frac{\phi^{2}(k)\delta^{\frac{\lceil\gamma_1 m\rceil-1}{\gamma_1}}}{\phi^{2}(\lceil\gamma_1 m\rceil)\delta^{\frac{k}{\gamma_1}}}=\delta^{\frac{\lceil\gamma_1 m\rceil-k}{\gamma_1}}
\frac{\phi^{2}(k)}{\phi^{2}(\lceil\gamma_1 m\rceil)}\delta^{-\frac{1}{\gamma_1}}\\
&\leq&\delta^{\frac{\lceil\gamma_1 m\rceil-k}{\gamma_1}}\cdot4^{\lceil\gamma_1 m\rceil-k}\cdot\delta^{-\frac{1}{\gamma_1}}
=(4\delta^\frac{1}{\gamma_1})^{\lceil\gamma_1 m\rceil-k}\cdot\delta^{-\frac{1}{\gamma_1}}\leq 4.
\end{eqnarray*}
In view of  (a) and (b), we have
\[|{\widehat\mu}(t)|^2\leq4\phi^{2}(\lceil\gamma_1 m\rceil).\]
Recall that
$$2^{m}\leq t<2^{m+1},$$
 we obtain
\begin{equation}\label{3.7}
|{\widehat\mu}(t)|\leq2\phi(\lceil\gamma_1\log_2|t|\rceil).
\end{equation}
This completes the proof $(1)$.
\item To prove $(2)$, we choose $K_{\phi}$ large enough so that
$$\pi^42^{-4K_{\phi}}+\pi^22^{-2K_{\phi}}\phi^{2}(1)+\pi^22^{-2(1+K_{\phi})}\phi^{2}(1)\leq 1,$$
In fact, we can choose
 $K_{\phi}=K_{\phi,2}:=3$.
  From the non-increasing property of $\phi$ and the assumption of $\dfrac{\phi(n)}{\phi(n+1)}\geq2$,
 \begin{eqnarray*}
  \prod_{i=1}^{2}(\phi^2(n_i)+\pi^22^{-2(l_i+K_{\phi})})&=&(\pi^22^{-2(l_1+K_{\phi})}+\phi^2(n_1))(\pi^22^{-2(l_2+K_{\phi})}+\phi^2(n_2)) \\
  &\leq&\left(\pi^42^{-4K_{\phi}}+\pi^22^{-2K_{\phi}}\phi^{2}(1)+\pi^22^{-2(1+K_{\phi})}\phi^{2}(1)\right)2^{-2(l_1+l_2)}\\
  &+&\phi^{2}(1+l_1+l_2)\\
  &\leq&2^{-2(l_1+l_2)}+\phi^2(1+l_1+l_2).
  \end{eqnarray*}

Then by induction,
 \begin{eqnarray}\label{3.4}
  \nonumber\prod_{i=1}^{j}(\phi^2(n_i)+\pi^22^{-2(l_i+K_{\phi})})&\leq&2^{-2(l_1+l_2+\cdots+l_j)}+\phi^2(1+l_1+l_2+\cdots +l_j)\\
   &=&2^{-2k}+\phi^{2}(1+k)\leq(1+\phi^{2}(1))2^{-2k}.
 \end{eqnarray}
Using Lemma \ref{lem3.1}, it is obvious that
 \begin{eqnarray}\label{3.5}
        \nonumber|{\widehat\mu}(t)|^2&\leq&\delta^{m-k-j}\prod_{i=1}^{j}(\phi^2(n_i)+\pi^22^{-2(l_i+K_{\phi})})\\
        &\leq&C_1\delta^{m-k-j}2^{-2k},
 \end{eqnarray}
 where $C_1=(1+\phi^{2}(1))$ is a constant depending on $\phi$.\\
 Similarity to \eqref{3.3}, we distinguish two cases.
Let $0<\gamma_2:=\frac{-\log_2\delta}{2}<1$, from \eqref{fg}, we have
 \begin{enumerate}[(a)]
   \item if $k\geq\gamma_2 m$, then
\[\delta^{m-k-j}2^{-2k}\leq\delta^{0}2^{-2\gamma_2 m}=2^{-2\gamma_2 m}.\]
   \item if $k<\gamma_2 m$, then
  \begin{eqnarray*}
   \delta^{m-k-j}2^{-2k}&\leq&\delta^{m-2k}2^{-2k}\leq\delta^{m-k/{\gamma_2}}2^{-2k}\\
  &\leq&2^{-2\gamma m}\frac{\delta^{m}2^{2\gamma_2 m}}{\delta^{k/{\gamma_2}}2^{2k}}\leq2^{-2\gamma_2 m}.
\end{eqnarray*}
\end{enumerate}
Combining (a) with (b), we get $$|{\widehat\mu}(t)|^2\leq C_2 2^{-2\gamma_2 m}.$$
Therefore we obtain that
\begin{equation}\label{3.6}
|{\widehat\mu}(t)|\leq C_2|t|^{-\gamma_2}.
\end{equation}
In view of \eqref{3.7} and \eqref{3.6}, the proof is complete.
\end{enumerate}
\textbf{Proof of theorem \ref{thm1.2}:} We distinguish two cases:
$\frac{\log b}{\log2}\in \mathbb{Q}$ and $\frac{\log b}{\log2}\notin \mathbb{Q}$ respectively.
\par Case one: $\frac{\log b}{\log2}\in \mathbb{Q}.$  We know that $x$ is normal for base $b$ if and only if $x$ is normal in base $2$, so it suffices
to prove that $\mu$ almost every $x$ is normal in base $2$, Since $\phi$ tends to $0$ at infinity, and from the aforementioned probability interpretation of the $\mu$, the random variables $X_n$ will have the same distribution asymptotically. By the strong law of large numbers, we can easily draw the conclusion.
\par Case two: $\frac{\log b}{\log2}\notin \mathbb{Q}$.  By Theorem \ref{del}, it suffices to prove that
$$\sum\limits^{N-1}_{m=0}\sum\limits^{N-1}_{n=0}|\widehat{\mu}(h(b^n-b^m))|=O(N^{2-\eta}),$$ for any non-zero integer $h$, where $0<\eta<1$ is a constant.
It is easy to see that
$$\sum\limits^{N-1}_{m=0}\sum\limits^{N-1}_{n=0}|\widehat{\mu}(h(b^n-b^m))|\leq 2 \sum\limits^{N-1}_{k=0}\sum\limits^{N-1}_{n=0}|\widehat{\mu}(h(b^k-1)b^n))|\leq 2N+2\sum\limits^{N-1}_{k=1}\sum\limits^{N-1}_{n=0}|\widehat{\mu}(h(b^k-1)b^n))|.$$
Thus, it is enough to prove
\begin{equation}\label{sblaa}
\sum\limits^{N-1}_{n=0}\mid\widehat{\mu}(h b^n)\mid=O(N^{1-\eta}),
\end{equation}
for any $h\in\mathbb{Z}\setminus\{0\}$.
\par Recalling \eqref{dks}, obverse that
\[|{\widehat\mu}(h2^\tau)|^2=\prod_{k=1}^{\infty}(\phi^2(k+\tau)\sin^2{\pi h2^{-k}}+\cos^2{\pi h2^{-k}}),\quad \text{for any nonnegative integer}\  \tau.\]
Assume that $b=b_02^{\tau_{0}}, h=h_02^{\tau_{1}}$, where $2\nmid b_0, 2\nmid h_0$, and $\tau_{0},\tau_{1}$ are two nonnegative integers,  then
\begin{eqnarray*}
    |{\widehat\mu}(hb^n)|^2&=&|{\widehat\mu}(h_02^{\tau_{1}}b_0^n2^{n\tau_{0}})|^2
     =\prod_{k=1}^{\infty}(\phi^2(k+\tau_{1}+n\tau_{0})\sin^2{\pi h_0b_0^n2^{-k}}+\cos^2{\pi h_0b_0^n2^{-k}})\\
     &\leq&\prod_{k=1}^{\infty}(\phi^2(k)\sin^2{\pi h_0b_0^n2^{-k}}+\cos^2{\pi h_0b_0^n2^{-k}})\\
       &\leq& |{\widehat\mu}(h_0b_0^n)|^2.
   \end{eqnarray*}
thus by \eqref{dk} and \eqref{sblaa}, it reduces to prove
\[\sum\limits^{N-1}_{n=0}\prod\limits_{k=1}^{\infty}\left(\frac{1}{2}(1+\phi(k))+\frac{1}{2}(1-\phi(k))e(h b^n 2^{-k})\right)=O(N^{1-\eta})\]
under the  condition that $2\nmid b,2\nmid h$.
\par First, we claim that for any $r\geq1$,
\begin{equation}\label{hgj}
\sum\limits^{2^r-1}_{n=0}\prod\limits_{k=1}^{\infty}\left(\frac{1}{2}(1+\phi(k))+\frac{1}{2}(1-\phi(k))e(h b^n 2^{-k})\right)=O(2^{r(1-\eta_0)}).
\end{equation}
\par Assume that $l$ is the largest integer such that $b\equiv 1\pmod{2^l}$, when $l=1$, note that $x$ is $b$ normal if and only if $x$ is $b^2$ normal, so we could choose $b=B^2$. If $l\geq2$, then
by the elementary number theory, $\{b^n,0\leq n\leq 2^r-1\}$ runs modulo $2^{l+r}$ through all residue classes which are congruent to $1$ modulo $2^{l}$.
Since $2\nmid h$, $\{hb^n,0\leq n\leq 2^r-1\}$ runs modulo $2^{l+r}$ through all residue classes which are congruent to $h$ modulo $2^{l}$. Thus
if we have the expansion
\[hb^n=\sum\limits_{k=0} d_k 2^k, d_k\in \{0,1\},\]
then r-strings
\begin{equation}\label{str}
d_{l}(n)d_{l+1}(n)\ldots d_{l+r-1}(n)
\end{equation}
take exactly one (and only one) of the $2^r$ possible sets of values as $n$ runs from $0$ to $2^r-1$ (this is first pointed out by Cassels, see \cite{Cass}).
 \par Now, we divide the indices $n(0\leq n\leq 2^r-1)$ into two cases, $\uppercase\expandafter{\romannumeral1}$ and $\uppercase\expandafter{\romannumeral2}$.
For convenience, setting $\mathcal{P}:=\{(0,1),(1,0)\}$ (we call $(0,1),(1,0)$ regular pairs). We say that the index $n$ belongs to case $\uppercase\expandafter{\romannumeral1}$ if in the string \eqref{str} there are at least $\varepsilon r(0<\varepsilon<\frac{1}{4})$ pairs of $(d_k,d_{k+1})$ belong to $\mathcal{P}$; otherwise $n$ is in case $\uppercase\expandafter{\romannumeral2}$.
\par For case $\uppercase\expandafter{\romannumeral1}$, if $(d_k,d_{k+1})\in \mathcal{P}$, then
$$\|hb^n2^{-k-2}\|\geq \frac{1}{4}.$$
Since case $\uppercase\expandafter{\romannumeral1}$ contains at most $2^r$ elements, by Corollary \ref{tl1}, we have
\begin{equation}\label{bds4}
\sum\limits_{n\in
\uppercase\expandafter{\romannumeral1}}\prod\limits_{k=1}^{\infty}\left(\frac{1}{2}(1+\phi(k))+\frac{1}{2}(1-\phi(k))e(h b^n 2^{-k})\right)\leq 2^r\delta_2^{\varepsilon r}=2^{r(1+\varepsilon\log\delta_2)}=2^{r(1-\eta_1)},
\end{equation}
where $\delta_2:=\frac{7}{8}+\frac{1}{8}\phi(1), \eta_1: =-\varepsilon\log(\frac{7}{8}+\frac{1}{8}\phi(1))\in (0,1)$.
\par For case $\uppercase\expandafter{\romannumeral2}$, we shall show that the number of $n$ such that there are less than $\epsilon r$ regular pairs in the string \eqref{str} is at most $2^{\gamma_3 r}$. It is easy to see that we only need to prove  the claim under the  condition that $k$ is odd. We know that in the string \eqref{str} the number of all combinations of the string such that there are exactly $s (s<\varepsilon r)$ indices belong to $\mathcal{P}$ is
$$\binom{\lceil r\rceil/2}{s}2^s2^{\lceil r\rceil/2-s}=\binom{\lceil r\rceil/2}{s}2^{\lceil r\rceil/2}.$$
Thus the number of $n$ for which there are less than $\varepsilon r$ regular pairs in the string \eqref{str} can not exceed
\begin{equation*}
  \sum\limits_{s=0}^{[\varepsilon r]}\binom{\lceil r\rceil/2}{s}2^{\lceil r\rceil/2}\leq \exp\big(h(\varepsilon) r)2^{\lceil r\rceil/2}\leq
  2^{(\frac{1}{2}+\rho(\varepsilon))r},
\end{equation*}
where $h(t)=-t\log t-(1-t)\log(1-t)$,  $\rho(\varepsilon)$ is a small constant which depending on $\varepsilon$, and the second inequality follows from Stirling's formula. Here $\rho(\varepsilon)\rightarrow 0$ as $\varepsilon\rightarrow 0$.
Consequently,
\begin{equation}\label{bds5}
\sum\limits_{n\in
\uppercase\expandafter{\romannumeral2}}\prod\limits_{k=1}^{\infty}|\left(\frac{1}{2}(1+\phi(k))+\frac{1}{2}(1-\phi(k))e(h b^n 2^{-k})\right)|\leq2^{(\frac{1}{2}+\rho(\varepsilon))r}:=2^{r(1-\eta_2)},
\end{equation}
where $\eta_2:=\frac{1}{2}-\rho(\varepsilon)$.
Together \eqref{bds4} with \eqref {bds5}, we obtain the claim \eqref{hgj}.
\par For general $N$, writing $N=2^{r_1}+2^{r_2}+\cdots+2^{r_k},r_1<r_2<\cdots<r_k$.
Divide $[1,N]$ into $k$ intervals $[0,2^{r_1}], [2^{r_1},2^{r_1}+2^{r_2}], \cdots, [2^{r_1}+2^{r_2}+\cdots+2^{r_{k-1}},N]$.
Note that each interval has the form of $[N_0,N_0+2^r]$ for some $r$, in view of \eqref{hgj},
\begin{equation}\label{ds1}
\sum\limits_{2^{r}+N_0> l\geq N_0}|\widehat{\mu}(h b^l)|=\sum\limits_{2^{r}> l\geq 0}|\widehat{\mu}(hb^{N_0} b^l)|=O(2^{r(1-\eta_0)})
=O(N^{(1-\eta_0)}).
\end{equation}
Clearly
$k\leq\log(N+1)-1$, thus
\begin{equation}\label{ds2}
\sum\limits^{N-1}_{l=0}|\widehat{\mu}(h b^l)|\leq N^{(1-\eta_0)}\log(N+1) =O(N^{1-\eta}).
\end{equation}
Combining \eqref{sblaa} with \eqref {ds2}, the proof is complete.
\section*{Acknowledgements} The authors would like to thank Aihua Fan, Xinggang He, Lingmin Liao, Baowei Wang, Shengyou Wen and Meng Wu for valuable suggestions. Thank Yuanyang Chang, Guotai Deng for assistances while the paper was being written. We also sincerely thank Teturo Kamae for careful reading and helpful comments on a preliminary version of this manuscript.

\end{document}